\documentclass[a4paper,11pt,reqno]{amsart}

\usepackage{fullpage}

\makeatletter \@addtoreset{equation}{section}

\makeatother \makeatletter

\usepackage[latin1]{inputenc}
\usepackage[english]{babel}
\usepackage{amsmath,amssymb,amsfonts,amsthm,mathrsfs,upgreek,accents}
\usepackage{url}
\usepackage{mathtools}
\usepackage{color}
\usepackage[colorlinks=true,linkcolor=blue,citecolor=blue,urlcolor=blue,breaklinks]{hyperref}
\usepackage{graphicx,color}
\usepackage{tikz}
\usepackage{pinlabel}

\usepackage{booktabs}
\usepackage{textcomp}
\usepackage{subfigure}

\numberwithin{equation}{section}

\theoremstyle{plain}
\begingroup
\newtheorem{theorem}{Theorem}[section]
\newtheorem{prop}[theorem]{Proposition}
\newtheorem{cor}[theorem]{Corollary}
\newtheorem{lemma}[theorem]{Lemma}
\endgroup

\theoremstyle{definition}
\begingroup

\newtheorem{remark}[theorem]{Remark}
\endgroup

\newcommand{\R}{\mathbb R}

\newcommand{\C}{\mathbb C}

\def\P{\mathcal P}

\DeclareMathOperator\supp{supp}

\def\div{\mathop{\mathrm{div}}}

\newcommand{\mres}{\mathbin{\vrule height 1.6ex depth 0pt width
0.13ex\vrule height 0.13ex depth 0pt width 1.3ex}}

\def\weakto{\rightharpoonup}

 \newcommand{\Rd}{\mathbb{R}^{d}}
 \newcommand{\Par}{\mathcal P}
 \newcommand{\ird}{\int_{\Rd}}

\title[The ellipse law]{The ellipse law: Kirchhoff meets dislocations}

\author[J.~A. Carrillo]{J.~A. Carrillo}
\author[J. Mateu]{J. Mateu}
\author[M.~G. Mora]{M.~G. Mora}
\author[L. Rondi]{L. Rondi}
\author[L. Scardia]{L. Scardia}
\author[J. Verdera]{J. Verdera}

\address[J.A. Carrillo]{Department of Mathematics, Imperial College London, United Kingdom}
\email{carrillo@imperial.ac.uk}

\address[J. Mateu]{Department de Matem\`atiques, Universitat Aut\`onoma de Barcelona, Catalonia}
\email{mateu@mat.uab.cat}

\address[M.G. Mora]{Dipartimento di Matematica, Universit\`a di Pavia, Italy}
\email{mariagiovanna.mora@unipv.it}

\address[L. Rondi]{Dipartimento di Matematica e Geoscienze, Universit\`a di Trieste, Italy}
\email{rondi@units.it}

\address[L. Scardia]{Department of Mathematical Sciences, University of Bath, United Kingdom}
\email{L.Scardia@bath.ac.uk}

\address[J. Verdera]{Department de Matem\`atiques, Universitat Aut\`onoma de Barcelona, Catalonia}
\email{jvm@mat.uab.cat}

\begin{document}

\begin{abstract}
In this paper we consider a nonlocal energy $I_\alpha$ whose
kernel is obtained by adding to the Coulomb potential an anisotropic
term weighted by a parameter $\alpha\in \R$. The
case $\alpha=0$ corresponds to purely logarithmic interactions,
minimised by the celebrated circle law for a quadratic confinement;
$\alpha=1$ corresponds to the energy of interacting dislocations,
minimised by the semi-circle law. We show that for $\alpha\in (0,1)$
the minimiser can be computed explicitly and is the normalised
characteristic function of the domain enclosed by an \emph{ellipse}.
To prove our result we borrow techniques from fluid dynamics, in particular those related to Kirchhoff's celebrated result that domains enclosed by ellipses are rotating vortex patches, called \emph{Kirchhoff ellipses}. Therefore we show a surprising connection between vortices and dislocations.

\bigskip


\medskip

\end{abstract}

\maketitle

\begin{section}{Introduction}
The starting point of our analysis is the nonlocal energy
\begin{equation}\label{ce}
I_{\alpha}(\mu) =\frac12 \iint_{\mathbb{R}^2\times \mathbb{R}^2}
W_{\alpha}(x-y) \,d\mu(x) \,d\mu(y) + \frac12\int_{\mathbb{R}^2}
|x|^2  \,d\mu(x)
\end{equation}
defined on probability measures $\mu\in \mathcal{P}(\R^2)$, where
the interaction potential $W_{\alpha}$ is given by
\begin{equation}\label{V:int0}
W_{\alpha}(x_1,x_2) = -\frac12 \log(x_1^2+x_2^2) +
\alpha\frac{x_1^2}{x_1^2+x_2^2}, \qquad x=(x_1,x_2)\in\R^2\,,
\end{equation}
and $\alpha\in\R$. Here the parameter $\alpha$ has the role of
\emph{tuning} the strength of the anisotropic component of
$W_\alpha$, making it more or less prominent.

In the particular case where the anisotropy is \emph{switched off},
namely for $\alpha=0$, the minimiser is radial, and is given by the
celebrated circle law $\mu_0:=\frac{1}\pi \chi_{B_1(0)}$,  the
normalised characteristic function of the unit disc. This result is
now classical and has been proved in a variety of contexts, from
Fekete sets to orthogonal polynomials, from random matrices to
Ginzburg-Landau vortices and Coulomb gases (see, e.g., \cite{Fro,
SaTo}, and the references therein).

In the case $\alpha=1$, the energy $I_1$ models interactions between
edge dislocations of the same sign (see, e.g., \cite{MPS, GPPS1}).
The minimisers of $I_1$ were since long conjectured to be vertical
walls of dislocations, and this has been confirmed only very
recently, in \cite{MRS}, where the authors proved that the only
minimiser of $I_1$ is the semi-circle law
\begin{equation}\label{semicirclaw}
\mu_{1} :=  \frac{1}{\pi}\delta_0\otimes \sqrt{2-x_2^2} \,
\mathcal{H}^1\mres(-\sqrt2,\sqrt2)
\end{equation}
on the vertical axis.

In this paper we explicitly characterise the minimiser of $I_\alpha$
for every $\alpha\in\R$. In particular, it turns out that the values
$\alpha=\pm1$ correspond to \emph{maximal anisotropy}. Increasing
the value of the weight $\alpha$ above $1$ has in fact no effect: a
simple energy comparison argument shows that $\mu_1$ is the only
minimiser of $I_\alpha$ for $\alpha\geq 1$. Moreover, the case
$\alpha<0$ can be recovered from the knowledge of the case
$\alpha>0$ by switching $x_1$ and $x_2$, so we can limit our
analysis to $\alpha\in(0,1)$.

For $\alpha\in (0,1)$ we prove that the unique minimiser of $I_\alpha$ is the normalised
characteristic function of the region surrounded by an ellipse of semi-axes $\sqrt{1-\alpha}$ and $\sqrt{1+\alpha}$.
 This shows, in particular, that $\alpha=1$ is a \emph{critical} value of the parameter, at which an abrupt change in the
 dimension of the support of the minimiser occurs.
The main result of the paper is the following:

\begin{theorem}\label{thm:chara}
Let $0\leq \alpha<1$. The measure
\begin{equation}\label{eq:sc-intro}
\mu_{\alpha} :=
\frac{1}{\sqrt{1-\alpha^2}\pi}\chi_{\Omega(\sqrt{1-\alpha},\sqrt{1+\alpha})},
\end{equation}
where
\begin{equation*}
\Omega(\sqrt{1-\alpha},\sqrt{1+\alpha}):=\left\{x=(x_1,x_2)\in\R^2:\
\frac{x_1^2}{1-\alpha}+\frac{x_2^2}{1+\alpha}<1\right\},
\end{equation*}
is the unique minimiser of the functional $I_\alpha$ among
probability measures $\P(\R^2),$ and satisfies the Euler-Lagrange
conditions
\begin{align}\label{EL-1-intro}
&(W_{\alpha}\ast \mu_{\alpha})(x) + \frac{|x|^2}2 = C_\alpha \quad
\text{for every } x\in \Omega(\sqrt{1-\alpha},\sqrt{1+\alpha}),
\\ \label{EL-2-intro}
&(W_{\alpha}\ast \mu_{\alpha})(x) + \frac{|x|^2}2 \geq C_\alpha
\quad \text{for every }x\in \R^2,
\end{align}
with
\begin{align*}
C_\alpha&= 2I_\alpha(\mu_\alpha)- \frac12\int_{\mathbb{R}^2} |x|^2 \,d\mu_\alpha(x)=\frac12- \log\Big(\frac{\sqrt{1-\alpha}+\sqrt{1+\alpha}}2\Big) + \alpha\, \frac{\sqrt{1-\alpha}}{\sqrt{1-\alpha}+\sqrt{1+\alpha}}\,.%\label{defcalpha}
\end{align*}
\end{theorem}

\smallskip

We emphasise that for $0\leq \alpha\leq 1$ the Euler-Lagrange
conditions \eqref{EL-1-intro}--\eqref{EL-2-intro} are a sufficient
condition to minimality, since we will show that the energy
$I_\alpha$ is strictly convex for these values of $\alpha$ (see
Proposition~\ref{exist+uniq}). Thus, proving that $\mu_\alpha$
satisfies \eqref{EL-1-intro}--\eqref{EL-2-intro} immediately entails
the minimality of $\mu_\alpha$.

\subsection{Kirchhoff ellipses and dislocations} To prove that the ellipse law $\mu_\alpha$ satisfies the Euler-Lagrange conditions  \eqref{EL-1-intro}--\eqref{EL-2-intro}, we evaluate the convolution of the kernel $W_\alpha$ with the characteristic function of the domain enclosed by a general ellipse. Let us define, for any $a,b>0$, the domain
\begin{equation*}%\label{ellipsedefin}
\Omega(a,b):=\left\{x=(x_1,x_2)\in\R^2:\
\frac{x_1^2}{a^2}+\frac{x_2^2}{b^2}<1\right\},
\end{equation*}
which is the region surrounded by an ellipse centred at the origin
with horizontal semi-axis $a$ and vertical semi-axis $b$. As a first step, we compute \emph{explicitly} the gradient of $W_\alpha\ast\chi_{\Omega(a,b)}$, both inside and outside
$\Omega(a,b)$; see  equations \eqref{gradinside}--\eqref{gradoutside} in
Proposition~\ref{gradientcomp}. As we shall see, this is enough to
conclude the proof of Theorem~\ref{thm:chara}, but for completeness
we shall also explicitly compute 
$W_\alpha\ast\chi_{\Omega(a,b)}$ in the whole plane (see Remark~\ref{rmk:pot-computation}). 
The gradient of $W_\alpha\ast\chi_{\Omega(a,b)}$ is the sum of $-(1/z)\ast\chi_{\Omega(a,b)}$, 
where $z=x_1+i x_2$ is the complex variable in the plane,
and of a second term containing the gradient of the anisotropic part of the potential. The convolution
$-(1/z)\ast\chi_{\Omega(a,b)}$ has been computed before, for instance in \cite{Joans}, for rotating vortex patches in fluid dynamics. 

Let us recall that
a vortex patch is the solution of the vorticity form of the planar Euler equations
in which the initial condition is the characteristic function of a bounded domain $D_0$. 
Since vorticity is transported by the flow, the vorticity at time $t$ is the characteristic function of a 
domain $D_t$. In general the evolution of $D_t$ is an extremely complicated phenomenon, but Kirchhoff 
proved more than one century ago that if $D_0$ is the domain enclosed by an ellipse with semi-axes $a$ and $b$,
then $D_t$ is just a rotation of $D_0$ around its centre of mass with constant angular velocity  
$\omega = ab/(a+b)$, see  \cite{Kirchoff,FP,MR}.  Domains with the simple evolution property described above are called $V$-states or rotating vortex patches. 
 They can be viewed as stationary solutions in a reference system that rotates with the patch, and they can be 
 described by means of an equation involving the stream function
 $-\log|\cdot|\ast \chi_{D_0}$ of the initial patch $D_0$ (see \cite{B}), which is formally similar to the 
 Euler-Lagrange equation \eqref{EL-1-intro}. If one wants to verify that for the elliptical patch $\Omega(a,b)$ such equation is satisfied, one needs to compute 
 explicitly $-\log|\cdot|\ast\chi_{\Omega(a,b)}$, and this can be done by first computing its gradient 
 $- (1/z)\ast \chi_{\Omega(a,b)}$.

The challenge in our case is computing the gradient of the anisotropic part of $W_\alpha \ast\chi_{\Omega(a,b)}$. 
The key observation is that it can be written in terms of suitable complex derivatives of the fundamental solution 
of the operator $\partial^2$, where $\partial = \partial/\partial z$. To compute such term explicitly we need the 
expression of $- (1/z)\ast \chi_{\Omega(a,b)}$, which was known, as well as the expression of $(z/\bar{z}^2) \ast \chi_{\Omega(a,b)}$, 
which we obtain in Proposition \ref{gradientcomp}.

What is surprising is that techniques developed in the context of fluid mechanics turn out to be crucial 
for the characterisation of the minimisers of the anisotropic energy $I_\alpha$, which arises, in the case $\alpha=1$, 
in the context of edge dislocations in metals. In particular the minimality of the semi-circle law for the dislocation energy $I_1$ 
can be deduced from Theorem~\ref{thm:chara} by a limiting argument based on
$\Gamma$-convergence (see Corollary~\ref{cor-alpha=1}). That is, we
obtain again the main result of \cite{MRS}, but with a different proof based on methods from fluid mechanics and
complex analysis.

It is worth emphasising the special role that ellipses play
in both contexts. On the one hand, in fluid mechanics they provide one of the few
explicit solutions of the incompressible Euler equations. On the other hand, the characteristic function 
of the elliptical domains $\Omega(a,b)$ 
is one of the few measures $\mu$ for which the convolution potential $W_{\alpha}\ast\mu$
 can be explicitly computed.

 What is even more surprising is that, for $0<\alpha<1$, the normalised characteristic function of 
 $\Omega(\sqrt{1-\alpha},\sqrt{1+\alpha})$  is actually
  the minimiser of the energy $I_{\alpha}$ and that it is possible to prove it.
  In fact, in the literature there are very few explicit characterisations of
minimisers for nonlocal energies and the only other example in the
nonradial case is the result proved in \cite{MRS} corresponding to
$\alpha=1$. The reason why the minimality of $\mu_{\alpha}$ is somewhat unexpected is the following. Let us first
consider the purely logarithmic case $\alpha=0$. By radial symmetry
of the energy and uniqueness, the minimiser $\mu_0$ must be radial.
This case is well-known to be connected to the classical obstacle
problem for the Laplace operator \cite{BK,CF,Ca}. Defining
$\Psi_0=W_0\ast \mu_0$ and assuming that $\mu_0$ is supported on the
closure of a smooth bounded open set $\Omega$, the
Euler-Lagrange equations \eqref{EL-1-intro}--\eqref{EL-2-intro} imply
\begin{equation}\label{eq:obstacle2l}
\begin{cases}
\displaystyle{\Psi_0 \geq C_0 - \frac{|x|^2}2}\quad & \text{ in } \R^2, \smallskip\\
-\Delta \Psi_0 \geq  0\quad & \text{ in } \R^2, \smallskip\\
\displaystyle{\Big(\Psi_0 -C_0 + \frac{|x|^2}2\Big) \Delta \Psi_0 = 0} & \text{ in  }
\R^2,
\end{cases}
\end{equation}
where $\overline{\Omega}$ is the coincidence set, i.e., the points where
$\Psi_0= C_0 - (|x|^2/2)$. It is not surprising from
\eqref{eq:obstacle2l} that $\mu_0$ is 
the normalised characteristic function of the coincidence set $\overline\Omega$ with constant density since $2\pi\mu_0=-\Delta \Psi_0$, and due to the radial symmetry the Euclidean ball is the clear candidate for $\Omega$.

In the presence of the anisotropic term, that is, for $\alpha>0$, we
write $W_\alpha=W_0+\alpha F$ and define  $\Psi_\alpha=W_\alpha\ast
\mu_\alpha$ where $\mu_\alpha$ is the unique minimiser of $I_\alpha$. A
corresponding obstacle problem as \eqref{eq:obstacle2l} can be
formally written for the potential $\Psi_\alpha$, and the coincidence
set is again determined by the condition $\Psi_\alpha= C_\alpha -
(|x|^2/2)$. Assuming it is the closure of a smooth bounded open set $\Omega$, one obtains
$$
\Delta \Psi_{\alpha}=-2\pi\mu_\alpha + \alpha \Delta F\ast
\mu_{\alpha}=-2 \quad \textrm{in } \Omega.
$$
If $\mu_{\alpha}$ is the normalised characteristic
function of $\Omega$, then $\Delta F\ast
\mu_{\alpha}$ should be constant on $\Omega$ as well.
However, computing $\Delta F\ast \chi_{\Omega}$ for a general domain
$\Omega$ is a highly non-trivial task, and in principle $\Delta F\ast \chi_{\Omega}$ 
could be a very complicated object. 
It is therefore
surprising that, for elliptic domains $\Omega=\Omega(a,b)$, $\Delta F\ast
\chi_{\Omega(a,b)}$ is constant in $\Omega(a,b)$. In fact, as we mentioned before, 
we are able to compute the convolution potential
$W_{\alpha}\ast \chi_{\Omega(a,b)}$ in the whole of $\R^2$, and to show
that  in $\Omega(a,b)$  it is a homogeneous polynomial of degree $2$ plus a 
constant. From this property, indeed, establishing
the first Euler-Lagrange condition is a relatively easy task. The
expression of the convolution potential (and of its gradient)
outside $\Omega(a,b)$ is instead much more involved, so that establishing
the second Euler-Lagrange condition is the challenge.

\subsection{Dimension of the support of the equilibrium measure} We have seen that the values $\alpha=\pm 1$ of the weight for the anisotropic term of the kernel $W_\alpha$ determine a sharp transition in the dimension of the support of the minimising measure $\mu_\alpha$ from two (for $\alpha\in (-1,1)$) to one (for $\alpha\leq-1$ and $\alpha\geq 1$).

For general energies of the form
\begin{equation} \label{eq:energy}
    E(\mu) = \frac 12 \ird\ird W(x-y) \, d\mu(x)\,d\mu(y) + \ird V(x) \, d\mu(x) \quad \text{for all $\mu \in \Par(\Rd)$},
\end{equation}
where $W\colon \Rd \to \R\cup\{+\infty\}$ is an \emph{interaction
potential} and $V\colon \Rd \to \R\cup\{+\infty\}$ is a
\emph{confining potential}, understanding how the dimension of the
support of the minimisers depends on $W$ and  $V$ is a challenging
question.

 In \cite{BCLR} the authors showed that the dimension of the support of a minimiser of $E$ is 
 directly related to the strength of the repulsion of the potential at the origin.
 What they showed is that the stronger the repulsion (up to Newtonian), the higher the dimension 
 of the support. The case of mild repulsive potentials in which the minimisers are finite number of
 Dirac deltas has been recently studied in \cite{CFP}.

Our result shows that a change of the dimension of the support of
the minimisers can also be obtained by tuning the asymmetry of the
interaction potential.

Another challenging question arising from the results in this paper
and in \cite{BCLR} is to give explicit examples in which a change of
the dimension of the support of the minimisers is obtained by tuning
the confining potential $V$,
%, or the asymmetry of the interaction potential,
or the singularity of the interaction potential at zero.

\subsection{More general interactions and evolution} The problem of analysing the landscape of energies
of the type
 \eqref{eq:energy} has triggered the attention of many analysts and applied mathematicians 
 in the last 20 years.

One of the main reasons for this interest, from the analytic
viewpoint, is that this question is directly linked to the stability
properties of stationary solutions of its associated gradient flow
\begin{equation}\label{gradflow}
    \partial_t\mu = \div \Big(\mu\, \nabla \frac{\delta E}{\delta \mu} \Big)= \div \big(\mu\, \nabla(W * \mu+V)\big) \quad \text{on $\Rd$, for $t>0$},
\end{equation}
in the Wasserstein sense \cite{CaMcCVi03,AmGiSa05}, where $\mu\colon
[0,\infty) \to \Par(\Rd)$ is a probability curve. Here, the
variational derivative $\delta E/\delta \mu:=W * \mu+V$ is obtained
by doing variations of the energy $E(\mu)$ preserving the unit mass
of the density as originally introduced in \cite{Otto}; see
\cite{villani,AmGiSa05} for the general theory. Equations like
\eqref{gradflow} describe the macroscopic behaviour of agents
interacting via a potential $W$, and are at the core of many
applications ranging from mathematical biology to economics; see
\cite{TBL,mogilner1999non,HP,BC} and the references therein.

In most of the early works, interaction and confinement potentials
were assumed to be smooth enough and convex in some sense, including
interesting cases with applications in granular media modelling
\cite{CaMcCVi03,Tos}. In most of the applications however the
potential $W$ is \emph{singular}, and in fact most of the rich
structure of the minimisers happens when the potentials are singular
at the origin; see \cite{d2006self,ring1,ring2,BCLR,ABCV,CH} and
\cite{CV} for a recent review in the subject. Typical interaction
potentials in applications are repulsive at the origin and
attractive at infinity (the latter guaranteeing confinement).

Euler-Lagrange necessary conditions for local minimisers of the
energy $E(\mu)$ in a suitable topology were derived in \cite{BCLR},
see also \cite{SaTo} for the particular case of the logarithmic
potential. They were used to give necessary and sufficient
conditions on repulsive-attractive potentials to have existence of
global minimisers \cite{CCP,SST}, and to analyse their regularity
for potentials which are as repulsive as, or more singular than, the
Newtonian potential \cite{CDM}. In both cases, global minimisers are
solutions of some related obstacle problems for Laplacian or
nonlocal Laplacian operators, implying that they are bounded and
smooth in their support, or even continuous up to the boundary
\cite{BK,CF,CDM,CV}. Similar Euler-Lagrange equations were also used
for nonlinear versions of the Keller-Segel model in order to
characterise minimisers of related functionals \cite{CCV}.

\bigskip

The plan of the paper is as follows. The proof of the Euler Lagrange
conditions in Theorem~\ref{thm:chara} will be done in
Section~\ref{proof-section}. We start next section,
Section~\ref{preliminary-section}, by showing the existence and
uniqueness of global minimiser for $I_\alpha$.
Section~\ref{second-proof-section} contains some additional
information. On the one hand, we discuss an alternative proof of the
first Euler-Lagrange condition and compute the minimal energy. On
the other hand, we study more general anisotropies.

\end{section}

%%%%%%%%%%%%%%%%%%%%%%%%%%%%%%%%%%%%%%%%%%%%

\begin{section}{Existence and uniqueness of the minimiser of $I_{\alpha}$}\label{preliminary-section}

In this section we prove that for every $\alpha\in \R$ the nonlocal
energy $I_{\alpha}$ defined in \eqref{ce} has a unique minimiser
$\mu_\alpha\in \mathcal{P}(\mathbb{R}^2)$, and that the minimiser
has a compact support.

We observe that it is sufficient to consider the case $\alpha\in
(0,1)$. In fact, for $\alpha=0$, that is, for purely logarithmic
interactions, it is well-known that there exists a unique minimiser
of $I_{0}$, which is given by the so-called circle law $\mu_0:=\frac
1\pi\chi_{B_1(0)}$ (see, e.g., \cite{Fro, SaTo}, and the references
therein). The case $\alpha=1$, that is, the case of interacting edge
dislocations, has been recently solved in \cite{MRS}, and it has
been shown that $I_1$ has a unique minimiser, given by the
semi-circle law \eqref{semicirclaw}. A simple comparison argument
shows that $\mu_1$ is indeed the unique minimiser of $I_{\alpha}$
for any $\alpha\geq 1$. In fact, for any $\alpha\geq 1$ and any
$\mu\in \mathcal{P}(\mathbb{R}^2)$ with $\mu\neq\mu_1$, we have
$$
I_{\alpha}(\mu_1)=I_{1}(\mu_1)<I_{1}(\mu)\leq I_{\alpha}(\mu).
$$
If $\alpha<0$, instead, we observe that
$$
W_{\alpha}(x_1,x_2)=-\log|x| + |\alpha|\frac{x_2^2}{|x|^2}+\alpha,
$$
hence all results in this case may be obtained from those with
$\alpha>0$ just by swapping $x_1$ and~$x_2$.

In what follows we assume the kernel $W_{\alpha}$ to be extended to
the whole of $\R^2$ by continuity, that is, we set
$W_{\alpha}(0):=+\infty$.

\begin{prop}\label{exist+uniq}
Let $\alpha\in [0,1]$. Then the energy $I_{\alpha}$ is well defined
on $\mathcal{P}(\mathbb{R}^2)$, is strictly convex on the class of
measures with compact support and finite interaction energy, and has
a unique minimiser in $\mathcal{P}(\mathbb{R}^2)$. Moreover, the
minimiser has compact support and finite energy.
\end{prop}

\begin{proof}
The case $\alpha=0$ is well-known. The proof for $\alpha\in (0,1)$
follows the lines of the analogous result for $\alpha=1$; see
\cite[Section~2]{MRS}. For the convenience of the reader we recall
the main steps of the proof.\smallskip

\noindent \textit{Step~1: Existence of a compactly supported
minimiser.} We have that
\begin{equation}
W_{\alpha}(x-y) + \frac12(|x|^2+|y|^2) \geq W_0(x-y) +
\frac12(|x|^2+|y|^2) \geq
 \left(\frac12-\frac{1}{e}\right)\, (|x|^2+|y|^2). \label{bound:compact}
\end{equation}
The lower bound \eqref{bound:compact} guarantees that $I_{\alpha}$
is well defined and nonnegative on $\mathcal{P}(\mathbb{R}^2)$ and,
since $I_\alpha(\mu_0)<+\infty$, where $\mu_0= \frac{1}{\pi}
\chi_{B_1(0)}$, it  implies that $\inf_{\mathcal{P}(\mathbb{R}^2)}
I_\alpha < +\infty$. It also provides tightness and hence
compactness with respect to narrow convergence for minimising
sequences, that, together with the lower semicontinuity of
$I_\alpha$, guarantees the existence of a minimiser.

As in \cite[Section~2.2]{MRS}, one can show that any minimiser of
$I_{\alpha}$ has compact support, again by
\eqref{bound:compact}.\smallskip

\noindent \textit{Step~2: Strict convexity of $I_\alpha$ and
uniqueness of the minimiser.} We prove that
\begin{equation}\label{cvx0}
\int_{\R^2} W_{\alpha}\ast (\nu_1-\nu_2) \,d(\nu_1-\nu_2) > 0
\end{equation}
for every $\nu_1, \nu_2 \in \mathcal{P}(\R^2)$, $\nu_1\neq \nu_2$,
with compact support and finite interaction energy, namely such that
$\int_{\R^2} (W_{\alpha}\ast\nu_i) \, d\nu_i < + \infty$ for
$i=1,2$. Condition \eqref{cvx0} implies strict convexity of
$I_{\alpha}$ on the set of probability measures with compact support
and finite interaction energy and, consequently, uniqueness of the
minimiser.

To prove \eqref{cvx0}, we argue again as in \cite[Section~2.3]{MRS}.
The heuristic idea is to rewrite the interaction energy of
$\nu:=\nu_1-\nu_2$ in Fourier space, as
\begin{equation*}%\label{Fourier:space}
\int_{\R^2} W_{\alpha}\ast \nu \,d \nu = \int_{\R^2}
\hat{W}_\alpha(\xi) |\hat\nu(\xi)|^2\, d\xi.
\end{equation*}
Since $\nu$ is a neutral measure, $\hat\nu$ vanishes at $\xi=0$. So,
the claim \eqref{cvx0} follows by showing the positivity of the
Fourier transform of $W_{\alpha}$ on positive test functions
vanishing at zero.

Since $W_{\alpha}\in L^1_{\mathrm{loc}}(\R^2)$ and has a logarithmic
growth at infinity, it is a tempered distribution, namely
$W_{\alpha}\in{\mathcal S}'$, where $\mathcal S$ denotes the
Schwartz space; hence $\hat W_{\alpha}\in{\mathcal S}'$. We recall
that $\hat W_{\alpha}$ is defined by the formula
$$
\langle \hat W_{\alpha}, \varphi\rangle := \langle W_{\alpha}, \hat
\varphi\rangle \qquad \text{ for every } \varphi\in{\mathcal S}
$$
where, for $\xi\in\R^2$,
\begin{equation*}
\hat \varphi(\xi):=\int_{\R^2}\varphi(x)e^{-2\pi i\xi\cdot x}\, dx.
\end{equation*}
Proceeding as in \cite[Section~2.3]{MRS}, we have that the Fourier
transform $\hat W_{\alpha}$ of $W_{\alpha}$ is given by
\begin{align}
\langle \hat W_{\alpha}, \varphi\rangle
 = \Big(\frac{\alpha}{2}+\gamma+\log\pi \Big) \varphi(0)
&+ \frac{1}{2\pi} \int_{|\xi|\leq 1}(\varphi(\xi)-\varphi(0))\frac{(1-\alpha)\xi_1^2+(1+\alpha)\xi_2^2}{|\xi|^4}\, d\xi \nonumber\\
&+ \frac{1}{2\pi}
\int_{|\xi|>1}\varphi(\xi)\frac{(1-\alpha)\xi_1^2+(1+\alpha)\xi_2^2}{|\xi|^4}
\,d\xi\label{hatV}
\end{align}
for every $\varphi\in{\mathcal S}$, where $\gamma$ is the Euler
constant. In particular, from \eqref{hatV}, we have that
\begin{equation}\label{transV0}
\langle \hat W_{\alpha}, \varphi \rangle =
\frac1{2\pi}\int_{\R^2}\frac{(1-\alpha)\xi_1^2+(1+\alpha)\xi_2^2}{|\xi|^4}\varphi(\xi)\,
d\xi
\end{equation}
for every $\varphi\in\mathcal S$ with $\varphi(0)=0$. Thus,
\eqref{transV0} implies that $\langle \hat W_{\alpha}, \varphi
\rangle>0$ for every $\varphi\in\mathcal S$ with $\varphi(0)=0$ and
$\varphi\geq0$, $\varphi\not\equiv0$.

Finally, the approximation argument in the proof of
\cite[Theorem~1.1]{MRS} allows one to pass from test functions in
$\mathcal S$ to measures. Hence \eqref{cvx0} is proved.
\end{proof}

%\begin{remark}[The ellipse rescaling]
%We point out that the expression \eqref{transV0} of $\hat
%W_{\alpha}$ suggests a rescaling of the space variables that leads
%to the right guess of the minimiser of $I_{\alpha}$ (see
%\eqref{eq:sc-intro}).
%\end{remark}

\end{section}

%%%%%%%%%%%%%%%%%%%%%%%%%%%%%%%%%

\begin{section}{Characterisation of the minimiser of $I_{\alpha}$: The ellipse law.}\label{proof-section}

It is a standard computation in potential theory (see \cite{SaTo,
MRS}) to show that any minimiser $\mu$ of $I_\alpha$ must satisfy
the following Euler-Lagrange conditions: there exists $C\in\R$ such
that
\begin{align}\label{EL-1-real}
&(W_{\alpha}\ast \mu)(x) + \frac{|x|^2}2 = C \quad  \text{for
$\mu$-a.e.\ } x\in \supp\mu,
\\ \label{EL-2-real}
&(W_{\alpha}\ast \mu)(x) + \frac{|x|^2}2 \geq C \quad \text{for
q.e.\ }x\in \R^2,
\end{align}
where {\em quasi everywhere} (q.e.) means up to sets of zero
capacity.  The Euler-Lagrange conditions
\eqref{EL-1-real}--\eqref{EL-2-real} are in fact equivalent to
minimality for $0\leq \alpha\leq 1$ due to
Proposition~\ref{exist+uniq}. See \cite[Section 3]{MRS}  for
details.

In this section we show that, for every $0\leq\alpha<1$, the measure
$\mu_\alpha$ defined in \eqref{eq:sc-intro} satisfies the
Euler-Lagrange conditions \eqref{EL-1-intro}--\eqref{EL-2-intro},
for some constant $C_\alpha\in\R$. By the above discussion  this
immediately implies that $\mu_{\alpha}$ is the unique minimiser of
$I_{\alpha}$, thus completing the proof of Theorem~\ref{thm:chara}.
The precise value of $C_\alpha$ will be computed in
Section~\ref{second-proof-section}.

We begin by studying $W_{\alpha}\ast \chi_{\Omega(a,b)}$ for every
$b\geq a>0$. We note that the function $W_\alpha\ast
\chi_{\Omega(a,b)}$ is $C^1$ in $\R^2$ (see \cite{Stein}). As a
first step, we compute the convolution $\big(\nabla W_{\alpha}\ast
\chi_{\Omega(a,b)}\big)(x)$ for every $b\geq a>0$ and at every point
$x\in \R^2$.

In fact we wish to prove that $\mu_\alpha$ satisfies the conditions
\begin{align}\label{EL-1-new}
\nabla(W_{\alpha}\ast \mu_{\alpha})(x) + x = 0  \quad &\text{for
every } x\in \Omega(\sqrt{1-\alpha},\sqrt{1+\alpha}),
\\ \label{EL-2-new}
x\cdot \nabla (W_{\alpha}\ast \mu_{\alpha})(x) + |x|^2 \geq 0 \quad
&\text{for every }x\in \R^2.
\end{align}
Clearly, conditions \eqref{EL-1-new}--\eqref{EL-2-new} imply that
\eqref{EL-1-intro}--\eqref{EL-2-intro} are satisfied for some
constant $C_\alpha\in\R$.

In order to evaluate the convolution $\nabla W_{\alpha}\ast
\chi_{\Omega(a,b)}$, it is convenient to work in complex variables.
As usual, we identify $z=x_1+ix_2\equiv x=(x_1,x_2)$, and we write
the standard differential operators as
$$
\partial= \frac{\partial}{\partial z}=\frac12 \left(\frac{\partial}{\partial x_1}-i\frac{\partial}{\partial x_2}\right)\quad \text{and} \quad \bar\partial= \frac{\partial}{\partial \bar z}=\frac12 \left(\frac{\partial}{\partial x_1}+i\frac{\partial}{\partial x_2}\right).
$$

In complex variables the potential $W_\alpha$ in \eqref{V:int0}
reads as
\begin{equation*}%\label{potcomp}
W_\alpha(x) \equiv W_\alpha (z)=-\frac12 \log (z\bar
z)+\frac{\alpha}2\left(1+\frac{z}{2\bar z}+\frac{\bar z}{2
z}\right),
\end{equation*}
and thus
\begin{equation}\label{gradpotcomp}
\nabla W_\alpha
(x)=-\frac{x}{|x|^2}+2\alpha\frac{x_1x_2}{|x|^4}x^\perp \equiv
2\bar\partial W_\alpha (z) =-\frac1{\bar
z}+\frac{\alpha}2\frac1{z}-\frac{\alpha}2\frac{z}{\bar z^2},
\end{equation}
where $x^\perp=(x_2,-x_1)$.

The result is the following.

\begin{prop}\label{gradientcomp}
Let $b\geq a>0$ and $\mu_{a,b}:=\frac{1}{\pi ab}\,
\chi_{\Omega(a,b)}$ be the (normalised) characteristic function of
the ellipse of semi-axes $a$ and $b$. Then we have
\begin{align}\nonumber
\nabla(W_{\alpha}\ast \mu_{a,b})(z)  &= \frac{1}{\pi
ab}\Big(-\frac1{\bar
z}+\frac{\alpha}2\frac1{z}-\frac{\alpha}2\frac{z}{\bar
z^2}\Big)\ast\chi_{\Omega(a,b)}\,(z) \\  &=
\frac1{ab}\left(-1-\alpha\lambda \right) z +
\frac1{ab}\left(\lambda+\frac{\alpha}2 +\lambda^2
\frac{\alpha}2\right) \bar z \label{gradinside}
\end{align}
for every $z\in\Omega(a,b)$ and
\begin{align}\nonumber
\nabla(W_{\alpha}\ast \mu_{a,b})(z)  &=\frac{1}{\pi ab}
\Big(-\frac1{\bar
z}+\frac{\alpha}2\frac1{z}-\frac{\alpha}2\frac{z}{\bar
z^2}\Big)\ast\chi_{\Omega(a,b)}\,(z)
\\
&= -(2+\alpha\lambda)h(\bar z)+\alpha h(z) - \alpha (\lambda\bar z -
z + 2ab h(\bar z)) h'(\bar z)  \label{gradoutside}
\end{align}
for every $z\in \Omega(a,b)^c$. Here
\begin{equation}\label{boundarycomp}
\lambda:=\frac{a-b}{a+b}, \qquad  h(z):=\frac{1}{z+\sqrt{z^2+c^2}},
\end{equation}
and $c^2=b^2-a^2$, where $c$ is the eccentricity of the ellipse.
\end{prop}

We note that here and in what follows $\sqrt{z^2+c^2}$ denotes the
branch of the complex square root that behaves asymptotically as $z$
at infinity. Namely, for $z\in\mathbb{C}\setminus [-ic,ic]$ such
that $z=\rho e^{i\theta}$ with $\rho>0$ and $0\leq \theta<2\pi$, we
have $z^2+c^2=\rho_1e^{i\theta_1}$ with $\rho_1>0$, and
 $0\leq \theta_1<2\pi$ if $0\leq \theta<\pi$ and $2\pi\leq \theta_1<4\pi$ if $\pi\leq \theta<2\pi$,
and $\sqrt{z^2+c^2}=\sqrt{\rho_1}e^{i\theta_1/2}$. In other words,
we choose the branch of the complex square root that preservs the
quadrants. In particular, for every $z\in \C\setminus [-ic, ic]$ we
have $\Re(z)\,\Re(\sqrt{z^2+c^2})\geq 0$ and
$\Im(z)\,\Im(\sqrt{z^2+c^2})\geq 0$, where $\Re(z)$ and $\Im(z)$
denote, respectively,  the real and imaginary part of $z$. This
property will be crucial in the proof of Theorem~\ref{thm:chara}.

\begin{proof}[Proof of Proposition~\ref{gradientcomp}]
We divide the proof into two steps.
\smallskip

\noindent \textit{Step~1: Computation of $\frac{1}{z}\ast
\chi_{\Omega(a,b)}$ and $\frac{1}{\bar z}\ast \chi_{\Omega(a,b)}$.}
We observe that $\frac{1}{z}\ast \chi_{\Omega(a,b)}$ is the Cauchy
transform of the (characteristic function of the) ellipse
$\Omega(a,b)$, up to a multiplicative constant. Indeed, the Cauchy
transform of a $C^1$ domain $\Omega\subset \C$ is defined as
\begin{equation}\label{cauchy}
\mathcal{C}(\chi_\Omega)(z):=\frac1{\pi}\int_\Omega
\frac1{z-\xi}\,d\xi.
\end{equation}
Clearly $\mathcal{C}(\chi_\Omega)$ is a continuous function in $\C$,
holomorphic in $\C\setminus\overline\Omega$ and vanishes at
infinity.

In the special case of an ellipse, namely for $\Omega =
\Omega(a,b)$, the expression \eqref{cauchy} can be computed
explicitly (see \cite[page 1408]{Joans}), and is given by
\begin{equation}\label{cauchye}
\frac1{\pi z}\ast \chi_{\Omega(a,b)}
=\mathcal{C}(\chi_{\Omega(a,b)})(z)= \begin{cases}
\bar z -\lambda z  &  \text{if }  z\in \Omega(a,b),\smallskip \\
2ab h(z) &  \text{if } z\in \Omega(a,b)^c,
\end{cases}
\end{equation}
where $\lambda$ and $h$ are as in \eqref{boundarycomp}.

By taking the conjugate of \eqref{cauchye} we obtain directly
\begin{equation}\label{cauchyeb}
\frac1{\pi \bar z}\ast
\chi_{\Omega(a,b)}=\mathcal{C}(\chi_{\Omega(a,b)})(\bar z) =
\begin{cases}
z -\lambda \bar z  &  \text{if }   z\in \Omega(a,b), \smallskip\\
2ab h(\bar z) &  \text{if }   z\in \Omega(a,b)^c,
\end{cases}
\end{equation}
and hence the first two terms of $\nabla W_\alpha \ast
\chi_{\Omega(a,b)}$ are now computed.\smallskip

\noindent \textit{Step~2: Computation of $\frac{z}{\bar z^2}\ast
\chi_{\Omega(a,b)}$.} We start by observing that
\begin{equation}\label{rel}
-\frac{1}{\pi} \frac{z}{\bar z^2}\ast \chi_{\Omega(a,b)} =\bar
\partial\left(\frac{1}{\pi} \frac{z}{\bar z} \ast \chi_{\Omega(a,b)}
\right),
\end{equation}
hence it is sufficient to compute $\frac{1}{\pi} \frac{z}{\bar z}
\ast \chi_{\Omega(a,b)}$. Now we recall that $\frac{1}{\pi}
\frac{z}{\bar z}$ is the fundamental solution of $\partial^2$, hence
$$
\partial^2 \left(\frac1\pi  \frac{z}{\bar z} \ast \chi_{\Omega(a,b)}\right) = \partial^2 \left( \frac1\pi \frac{z}{\bar z} \right)\ast \chi_{\Omega(a,b)}=\begin{cases}
1  &  \text{if } z\in \Omega(a,b),\\
0 &  \text{if } z\in \Omega(a,b)^c.
\end{cases}
$$
The previous expression implies that $\frac1\pi  \frac{z}{\bar z}
\ast \chi_{\Omega(a,b)}$ is quadratic in $z$ in $\Omega(a,b)$ and
linear in $z$ in $\Omega(a,b)^c$. More precisely,
\begin{equation}\label{new1}
\frac1\pi  \frac{z}{\bar z}\ast \chi_{\Omega(a,b)}= \begin{cases}
\frac{z^2}2 +z h_1^i(\bar z) + h_2^i(\bar z) &  \text{if }  z\in \Omega(a,b),\smallskip\\
z h_1^o(\bar z) + h_2^o(\bar z) &  \text{if }  z\in \Omega(a,b)^c,
\end{cases}
\end{equation}
with $h_1^i$,  $h_2^i$,  $h_1^o$ and  $h_2^o$ holomorphic functions
in their respective domains (and the indices $i$ and $o$ stand for
``inner" and ``outer").

It remains to determine the functions $h_1^i$,  $h_2^i$,  $h_1^o$
and  $h_2^o$ explicitly. By applying the operator $\partial$ to both
sides of \eqref{new1}, we deduce
$$
\frac{1}{\pi\bar z}\ast \chi_{\Omega(a,b)}=
 \begin{cases}
z + h_1^i(\bar z) &  \text{if } z\in \Omega(a,b),\smallskip\\
 h_1^o(\bar z) &  \text{if } z\in \Omega(a,b)^c,
\end{cases}
$$
which, together with \eqref{cauchyeb}, leads to the identification
of $h_1^i$ and $h_1^o$, as
$$
h_1^i(\bar z)=-\lambda \bar z \text{ in } \Omega(a,b) \quad \text{
and } \quad h_1^o(\bar z)=2ab h(\bar z) \text{ in } \Omega(a,b)^c
\,.
$$
Substituting these expressions into \eqref{new1} we then have
\begin{equation}\label{new12}
\frac1\pi  \frac{z}{\bar z}\ast \chi_{\Omega(a,b)}= \begin{cases}
\frac{z^2}2 -\lambda \bar z z + h_2^i(\bar z)  &\text{if }  z\in \Omega(a,b),\smallskip\\
2ab z h(\bar z) + h_2^o(\bar z)  &\text{if }  z\in \Omega(a,b)^c,
\end{cases}
\end{equation}
with $h_2^i$ and  $h_2^o$ holomorphic functions in their respective
domains, still to be determined. By \eqref{rel}, however, it is
sufficient to determine their derivatives, since, by applying the
operator $\bar \partial$ to both sides of \eqref{new12}, we have
\begin{equation}\label{new2}
-\frac1\pi \frac{z}{\bar z^2}\ast \chi_{\Omega(a,b)}=\bar
\partial\left(\frac1\pi \frac{z}{\bar z} \right)\ast
\chi_{\Omega(a,b)} = \begin{cases}
-\lambda z + (h_2^i)' (\bar z)  &\text{if }  z\in \Omega(a,b), \smallskip \\
2ab z h'(\bar z) + (h_2^o)'(\bar z)  &\text{if }  z\in
\Omega(a,b)^c.
\end{cases}
\end{equation}
Now we observe that the function $\frac1\pi \frac{z}{\bar z^2}\ast
\chi_{\Omega(a,b)}$ on the left-hand side of \eqref{new2} is
continuous in $\C$  and decays to zero as $z\to \infty$; see e.g.\
\cite{Stein}. Therefore, also the right-hand side of \eqref{new2} is
continuous in $\C$, which implies in particular that
\begin{equation}\label{equality:boundary}
-\lambda z + (h_2^i)' (\bar z) = 2ab z h'(\bar z) + (h_2^o)'(\bar z)
\end{equation}
for every $z\in \partial \Omega(a,b)$. By using the expression of
the boundary of the ellipse in complex variables, namely
\begin{equation*}
\partial \Omega(a,b)=\{z\in \C : \bar z = \lambda z + 2ab h(z)\}
\end{equation*}
where $\lambda$ and $h$ are defined as in \eqref{boundarycomp}, and
by rearranging the terms in \eqref{equality:boundary}, we obtain
that
\begin{equation}\label{relbder}
-\lambda^2 \bar z + (h_2^i)' (\bar z) = 2ab\lambda h(\bar z)+2ab
(\lambda\bar z+2ab h(\bar z)) h'(\bar z) + (h_2^o)'(\bar z)
\end{equation}
on $\partial \Omega(a,b)$. Consider now the auxiliary function
\begin{equation}\label{new3}
R(\bar z)=\begin{cases}
-\lambda^2 \bar z + (h_2^i)' (\bar z)  &  \text{if }  z\in \Omega(a,b),\smallskip\\
2ab\lambda h(\bar z)+2ab (\lambda\bar z+2ab h(\bar z)) h'(\bar z) +
(h_2^o)'(\bar z)  &  \text{if }  z\in \Omega(a,b)^c.
\end{cases}
\end{equation}
Because of the continuity condition \eqref{relbder}, $R(\bar z)$ is
an anti-holomorphic function in $\C$. Moreover, it easy to see that
$R$ has zero limit at $\infty$. This is clear for all the terms in
the expression of $R$ in $\Omega(a,b)^c$ involving $h$ and $h'$, by
\eqref{boundarycomp}; for the term $(h_2^o)'$ it follows by
\eqref{new2}. The Liouville Theorem then implies that $R(\bar
z)\equiv 0$. As a consequence, both expressions on the right-hand
side of \eqref{new3} are zero, which gives
\begin{equation*}%\label{h2der}
(h_2^i)'(\bar z)=\lambda^2 \bar z \text{ in } \Omega(a,b) \quad
\text{ and } \quad (h_2^o)'(\bar z)=-2ab\lambda h(\bar z)-2ab
(\lambda\bar z+2ab h(\bar z)) h'(\bar z) \text{ in } \Omega(a,b)^c,
\end{equation*}
and hence the identification of $(h_2^i)'$ and $(h_2^o)'$ in their
respective domains.

Plugging these formulas into \eqref{new2}, we finally conclude that
\begin{equation}\label{new4}
-\frac1\pi \frac{z}{\bar z^2}\ast \chi_{\Omega(a,b)} = \begin{cases}
\lambda^2 \bar z - \lambda z &  \text{if }  z\in \Omega(a,b),\smallskip\\
 2ab (z-\lambda\bar z - 2ab h(\bar z)) h'(\bar z) - 2ab\lambda h(\bar z) & \text{if }  z\in \Omega(a,b)^c.
\end{cases}
\end{equation}

\smallskip

Finally, using \eqref{cauchye}, \eqref{cauchyeb} and \eqref{new4} we
have that \eqref{gradinside} and \eqref{gradoutside} immediately
follow.
\end{proof}

We are now in the position to prove our main result.

\begin{proof}[Proof of Theorem~\ref{thm:chara}]
We need to show that \eqref{EL-1-new}--\eqref{EL-2-new} hold. By
using the expression \eqref{gradpotcomp}, conditions
\eqref{EL-1-new}--\eqref{EL-2-new} are equivalent to show that for
$a=\sqrt{1-\alpha}$ and $b=\sqrt{1+\alpha}$ we have
%for the normalised characteristic function of a general ellipse, as
\begin{eqnarray}\label{EL-1-comp}
\frac{1}{\pi ab} \big(\nabla W_{\alpha}\ast\chi_{\Omega(a,b)}\big) \,(z) + z =0 && \text{ for every } z\in \Omega(a,b),\\
\label{EL-2-comp} \frac{1}{\pi ab}\Re\Big(\bar z \big(\nabla
W_{\alpha}\ast\chi_{\Omega(a,b)}\big) \,(z) \Big)  + |z|^2 \geq 0 &&
\text{ for every } z\in \Omega(a,b)^c.
\end{eqnarray}
\smallskip

\noindent \textit{Step~1: The measure $\mu_\alpha$ satisfies
\eqref{EL-1-new}.} Using \eqref{gradinside}, we have that
$$
\frac{1}{\pi ab}\big(\nabla W_{\alpha}\ast
\chi_{\Omega(a,b)}\big)(z) + z = \frac1{ab}\left(-1-\alpha\lambda+ab
\right) z + \frac1{ab}\left(\lambda+\frac{\alpha}2 +\lambda^2
\frac{\alpha}2\right) \bar z
$$
for every $z\in\Omega(a,b)$. It is easy to check that
$a=\sqrt{1-\alpha}$ and $b=\sqrt{1+\alpha}$ are the unique solution
of the system
$$
\begin{cases}
-1-\alpha\lambda+ab=0, \smallskip\\
\lambda+\dfrac{\alpha}2 +\lambda^2 \dfrac{\alpha}2=0,
\end{cases}
$$
leading to condition \eqref{EL-1-comp}, hence \eqref{EL-1-new}.

\smallskip

\noindent \textit{Step~2: The measure $\mu_\alpha$ satisfies
\eqref{EL-2-new}.} By \eqref{gradoutside} we have that
\begin{equation}\label{defg}
(\nabla W_{\alpha}\ast \mu_{\alpha})(z) + z =
-(2+\alpha\lambda)h(\bar z)+\alpha h(z) - \alpha (\lambda\bar z - z
+ 2ab h(\bar z)) h'(\bar z) + z
\end{equation}
for every $z\in \Omega(\sqrt{1-\alpha},\sqrt{1+\alpha})^c$. Since
$a=\sqrt{1-\alpha}$ and $b=\sqrt{1+\alpha}$, we note that
\begin{equation}\label{lambda_ab}
\lambda = \frac{1}{\alpha}\,(\sqrt{1-\alpha^2}-1) \qquad \text{and}
\qquad ab = \sqrt{1-\alpha^2}.
\end{equation}
To simplify the expression \eqref{defg} we also observe that
\begin{equation}\label{hhprime}
h(z)=\frac{1}{z+\sqrt{z^2+2\alpha}}=\frac1{2\alpha}(\sqrt{z^2+2\alpha}-z)
\qquad \text{and} \qquad  h'(z)=\frac{-h(z)}{\sqrt{z^2+2\alpha}},
\end{equation}
where we have used the fact that $c^2 = b^2-a^2 = 2\alpha$.
Substituting \eqref{lambda_ab} and \eqref{hhprime} into
\eqref{defg}, and performing some simple algebraic manipulations, we
deduce that
\begin{align*}%\label{finalnablaout}
(\nabla W_{\alpha}\ast \mu_{\alpha})(z) + z &= \frac12 \sqrt{ z^2+2\alpha} -\frac1{2\alpha} \sqrt{\bar z^2+2\alpha} + \frac{\bar z}2 \Big(z+\frac1{\alpha}\bar z\Big)\frac{1}{\sqrt{\bar z^2+2\alpha}} \nonumber\\
& = \frac{|z^2+2\alpha|+|z^2|-2}{2|z^2+2\alpha|}\,
\sqrt{z^2+2\alpha}.
\end{align*}
Proving that \eqref{EL-2-comp} holds with $a=\sqrt{1-\alpha}$ and
$b=\sqrt{1+\alpha}$ is then equivalent to showing that
\begin{equation}\label{claim_foci}
\frac{|z^2+2\alpha|+|z^2|-2}{2|z^2+2\alpha|} \, \Re\big(\bar
z\sqrt{z^2+2\alpha}\big) \geq 0 \quad \text{for every } z\in
\Omega(\sqrt{1-\alpha}, \sqrt{1+\alpha})^c.
\end{equation}
Now, we recall that $\sqrt{z^2+2\alpha}$ denotes the branch of the
complex square root preserving the quadrants, that is, for every
$z\in \C\setminus [-i\sqrt{2\alpha}, i\sqrt{2\alpha}]$ we have
$\Re(z)\,\Re(\sqrt{z^2+2\alpha})\geq 0$ and
$\Im(z)\,\Im(\sqrt{z^2+2\alpha})\geq 0$. Therefore, we immediately
deduce that $\Re(\bar z \sqrt{z^2+2\alpha})\geq 0$ for every $z\in
\Omega(\sqrt{1-\alpha}, \sqrt{1+\alpha})^c$.

To conclude the proof of the claim \eqref{claim_foci} it remains to
show that $|z^2+2\alpha|+|z^2|-2\geq 0$ in
$\Omega(\sqrt{1-\alpha},\sqrt{1+\alpha})^c$. This is true since
$|z^2+2\alpha|+|z|^2-2$ is a level-set function for the ellipse
$\Omega(\sqrt{1-\alpha},\sqrt{1+\alpha})$. This is a general
statement for ellipses $\Omega(a,b)$ with $b\geq a>0$, that we prove
in Lemma~\ref{lemma:foci} below.

The proof of Theorem~\ref{thm:chara} is thus complete, up to the
computation of the constant $C_\alpha$, that we postpone to
Section~\ref{second-proof-section}.
\end{proof}

\begin{lemma}\label{lemma:foci}
Let $\Gamma=\partial\Omega(a,b)$, with $b\geq a>0$. Then
 \begin{align}
 & |z^2| + |z^2+c^2| = a^2+b^2 \quad \text{ if } z \in \Gamma, \label{ellipse}\\
& |z^2| + |z^2+c^2| \ge  a^2+b^2 \quad \text{ if } z \in
\Omega(a,b)^c, \label{ellipseout}
 \end{align}
where $c^2 = b^2-a^2$.
\end{lemma}

\begin{proof}
By dilating $z$ by a factor of $\frac{2}{a+b}$ we can further assume
that $a+b=2$, and write $a= 1-\beta$ and $b= 1+\beta$, for some $0
\le \beta < 1$. Thus $c^2 = (1+\beta)^2 -(1-\beta)^2= 4 \beta$ and
$a^2+b^2= (1-\beta)^2+ (1+\beta)^2 = 2 (1+\beta^2)$, and the claim
becomes
\begin{align}
& |z^2| + |z^2+4\beta| = 2 (1+\beta^2) \quad \text{ if } z \in \Gamma, \label{ellipse1}\\
& |z^2| + |z^2+4\beta| \ge  2 (1+\beta^2) \quad \text{ if } z \in
\Omega(1-\beta,1+\beta)^c. \label{ellipseout1}
 \end{align}
Since  $ \Gamma = \{z= \zeta - \beta \bar{\zeta} : |\zeta |= 1 \}$,
we have
\begin{equation}\label{z2}
z^2 = (\zeta^2 + \beta^2 \bar{\zeta}^2) - 2 \beta
\end{equation}
and
\begin{equation}\label{zc2}
z^2 + 4 \beta = (\zeta^2 + \beta^2 \bar{\zeta}^2) + 2 \beta.
\end{equation}
Now we observe that, whenever $\zeta\in \C$, $|\zeta|=1$, then
$\zeta^2 + \beta^2 \bar{\zeta}^2 \in \partial
\Omega(1+\beta^2,1-\beta^2)$. Hence, since the foci of the ellipse
$\partial \Omega(1+\beta^2,1-\beta^2)$ are $\pm 2 \beta$, we deduce
by \eqref{z2}--\eqref{zc2} that
$$
|z|^2 + |z^2+ 4 \beta| = 2(1+\beta^2),
$$
which completes the proof of \eqref{ellipse1} (and then of
\eqref{ellipse}). The statement \eqref{ellipseout} can be proved in
the same way.
\end{proof}

The limiting case $\alpha=1$ studied in \cite{MRS} can be obtained
from our analysis, valid for $0\leq \alpha<1$, by means of a simple
argument based on $\Gamma$-convergence. As a first step, we note
that $(I_\alpha)_{\alpha\in (0,1)}$ is an increasing family of lower
semicontinuous functionals (with respect to the narrow convergence
of measures). Hence $I_1$ is not only the pointwise limit of
$I_\alpha$ as $\alpha\to 1^-$, but also the $\Gamma$-limit, namely
$$
\Gamma\text{-}\!\!\!\lim_{\alpha\to 1^-} I_\alpha = I_1,
$$
see, e.g., \cite[Proposition 5.4]{DM}. Let now $\mu_\alpha$ and
$\mu_1$ be the measures defined in \eqref{eq:sc-intro} and in
\eqref{semicirclaw}, respectively. Since $\mu_\alpha$ is a minimiser
of $I_\alpha$ for every $\alpha\in (0,1)$, and since
$\mu_\alpha\weakto \mu_1$ narrowly as $\alpha\to 1^-$, the
Fundamental Theorem of $\Gamma$-convergence implies that $\mu_1$ is
a minimiser of $I_1$. It is in fact the unique minimiser, by the
strict convexity of $I_1$.

\begin{cor}\label{cor-alpha=1}
The unique minimiser of $I_1$ is given by the semi-circle law
$$
\mu_{1} :=  \frac{1}{\pi}\delta_0\otimes \sqrt{2-x_2^2} \,
\mathcal{H}^1\mres(-\sqrt2,\sqrt2)\,.
$$
\end{cor}

\end{section}

%%%%%%%%%%%%%%%%%%%%%%%%%%%%%%%%%%%%%

\begin{section}{Further Comments}\label{second-proof-section}

\subsection{Stationarity of $\mu_\alpha$: an alternative proof.}

Here, we provide an alternative proof of the fact that, for every
$\alpha\in (0,1)$, the ellipse-law $\mu_{\alpha}$ in
\eqref{eq:sc-intro} is a stationary solution of the gradient flow
\eqref{gradflow} associated to \eqref{EL-1-intro}, namely it
satisfies the Euler-Lagrange condition \eqref{EL-1-intro} inside its
support, for some constant $C_\alpha\in\R$. In doing so, we also
compute the exact value of $C_\alpha$ and the minimum value of
$I_\alpha$. Finally, we  explicitly compute the function
$W_{\alpha}\ast\mu_{a,b}$ in the whole of $\mathbb{R}^2$ for a
general ellipse (see Remark~\ref{rmk:pot-computation}). We recall
that $\mu_{a,b}= \frac{1}{\pi ab} \, \chi_{\Omega(a,b)}$, with
$0<a\leq b$, is the (normalised) characteristic function of the
ellipse of semi-axes $a$ and $b$.

The proof we propose in this section uses the explicit expression of
the logarithmic potential of $\mu_{a,b}$, namely $-\log|\cdot|\ast
\mu_{a,b}$, which is well-known in the literature, in the context of
fluid mechanics. This potential represents the stream function
associated to the vorticity corresponding to an elliptic vortex
patch (the Kirchhoff ellipse) rotating with constant angular
velocity about its centre, and it was computed in order to prove
that the Kirchhoff ellipses are V-states of the Euler equations in
two dimensions \cite{Kirchoff,FP,MR,Joans}.

The \textit{explicit} expression of the logarithmic potential for
any ellipse
\begin{equation}\label{Phi_ellipse}
\Phi_{a,b}:= -\log|\cdot|\ast \mu_{a,b}
\end{equation}
is well-known and is given by
\begin{equation*}
\Phi_{a,b}(x)=
\begin{cases}
\medskip
\displaystyle -\frac{1}{ab}\, \frac{bx_1^2 + ax_2^2}{a+b} - \log\Big(\frac{a+b}2\Big)+\frac12 \quad &\text{if } x \in \Omega(a,b),\\
H(x) \quad &\text{if } x \in \Omega(a,b)^c,
\end{cases}
\end{equation*}
where the function $H$ is defined as
$$
H(x)\equiv H(z)=\begin{cases}
\medskip
\displaystyle %-\frac 12\log(x_1^2+x_2^2)=
- \log |z| \quad &\text{if } a=b,\smallskip\\
\displaystyle -\frac 1{c^2}\Re \big(z\sqrt{z^2+c^2}-z^2\big)-
\log|\sqrt{z^2+c^2}+z| +\log 2+\frac 12
%- \log\Big(\frac{\sqrt{b^2-a^2}}2\Big)
\quad &\text{if } a<b.
\end{cases}
$$
We note that $H$ is real-valued, $H(z)=H(\bar{z})$, and that
\begin{equation}\label{nablaH0}
\nabla H(x)\equiv 2 \bar \partial H(\bar z)=-\frac 1{\bar z}
\end{equation}
if $a=b$, whereas
\begin{equation}\label{nablaH}
\nabla H(x)\equiv 2 \bar \partial H(\bar z)=-2h(\bar z)=-\frac
2{c^2}(\sqrt{\bar z^2+c^2}-\bar z)
\end{equation}
for $a<b$.

Note that $\Phi_{a,b}$ is only one part (the radial component) of
the convolution potential $W_\alpha\ast \mu_{a,b}$. We now show that
the anisotropic part of $W_\alpha\ast \mu_{a,b}$ can be obtained
from $\Phi_{a,b}$ by means of an ingenious differentiation. We first
write \eqref{Phi_ellipse} explicitly, for $x\in \Omega(a,b)$ and
$0<a<b$:
\begin{equation}\label{Phiabx}
\Phi_{a,b}(x) = -\frac{1}{\pi ab} \int_{\Omega(a,b)}\log|x-y| \,dy =
-\frac{1}{ab}\, \frac{bx_1^2 + ax_2^2}{a+b} -
\log\Big(\frac{a+b}2\Big)+\frac12.
\end{equation}
We perform a change of variables in order to write the integral in
the expression above as an integral on the fixed domain $B_1(0)$,
the unit disc. In terms of the new variables $u=(u_1,u_2) :=
\left(\frac{x_1}{a}, \frac{x_2}{b}\right)$, $v=(v_1,v_2) :=
\left(\frac{y_1}{a}, \frac{y_2}{b}\right)$, and the aspect ratio
$k:=a/b$, $k\in (0,1)$, the expression in \eqref{Phiabx} becomes
\begin{equation}\label{Phi-ball}
-\log b - \frac1{2\pi}
\int_{B_1(0)}\log\big(k^2(u_1-v_1)^2+(u_2-v_2)^2\big)\,dv =
-\frac{ku_1^2+u_2^2}{1+k} - \log\frac b2 - \log(1+k) +\frac12.
\end{equation}
By differentiating the previous expression \eqref{Phi-ball} with
respect to the aspect ratio $k$ we obtain the identity
\begin{equation*}
\frac1\pi
\int_{B_1(0)}\frac{k(u_1-v_1)^2}{k^2(u_1-v_1)^2+(u_2-v_2)^2}\,dv =
\frac{u_1^2 -u_2^2}{(1+k)^2} + \frac{1}{1+k},\quad k\in (0,1),
\end{equation*}
which, expressed in the original variables $x$ and $y$, and $a,b$,
becomes
\begin{equation}\label{Phi-prime}
\frac1{\pi ab} \int_{\Omega(a,b)}\frac{(x_1-y_1)^2}{|x-y|^2}\,dy =
\frac{a}{a+b} + \frac{b^2x_1^2 - a^2 x_2^2}{ab(a+b)^2}.
\end{equation}
Note that the left-hand side of \eqref{Phi-prime} is exactly the
convolution of the anisotropic term of the potential $W_\alpha$ with
the measure $\mu_{a,b}$. This allows us to compute the whole
convolution potential $W_\alpha\ast \mu_{a,b}$ on $\Omega(a,b)$:
\begin{align}\label{Vastmu:m}
(W_\alpha\ast \mu_{a,b})(x) &= \frac{1}{\pi ab}
\int_{\Omega(a,b)}\Big(-\log|x-y|
+\alpha \frac{(x_1-y_1)^2}{|x-y|^2}\Big)\, dy \nonumber \\
& = -\frac{1}{ab}\, \frac{bx_1^2 + ax_2^2}{a+b} - \log\Big(\frac{a+b}2\Big)+\frac12 + \alpha\, \frac{a}{a+b} + \alpha\,\frac{b^2x_1^2 - a^2 x_2^2}{ab(a+b)^2} \nonumber\\
& = \frac{-a-b+\alpha \,b}{a(a+b)^2}\, x_1^2 -  \frac{a+b+\alpha
\,a}{b(a+b)^2}\, x_2^2 - \log\Big(\frac{a+b}2\Big)+\frac12 +
\alpha\, \frac{a}{a+b}.
\end{align}
Then we can evaluate the value of the energy $I_\alpha$ on ellipses
$\mu_{a,b}$, namely
\begin{align}
I_\alpha(\mu_{a,b}) &= \frac{1}{2\pi ab} \int_{\Omega(a,b)} (W_\alpha\ast \mu_{a,b})(x) \,d x +  \frac1{2 \pi ab}\int_{\Omega(a,b)} (x_1^2+ x_2^2) \,dx\nonumber\\
&= \frac{1}{2\pi ab}\Big(1+\frac{-a-b+\alpha \,b}{a(a+b)^2}\Big) \int_{\Omega(a,b)} x_1^2 \,dx + \frac{1}{2\pi ab}\Big(1 -  \frac{a+b+\alpha \,a}{b(a+b)^2}\Big)  \int_{\Omega(a,b)} x_2^2 \,dx \nonumber\\
&\quad - \frac12\log\Big(\frac{a+b}2\Big)+\frac14 + \frac{\alpha}2\,
\frac{a}{a+b}, \label{energy:ellipse}
\end{align}
where
\begin{equation}\label{confinement:e}
\frac1{\pi ab}\int_{\Omega(a,b)} x_1^2 \,dx = \frac{a^2}4 \qquad
\text{and} \qquad  \frac1{\pi ab}\int_{\Omega(a,b)} x_2^2 \,dx =
\frac{b^2}4.
\end{equation}

In the special case of $a=\sqrt{1-\alpha}$ and $b=\sqrt{1+\alpha}$
we have that
$$
\frac{-a-b+\alpha \,b}{a(a+b)^2} = -\frac{a+b+\alpha \,a}{b(a+b)^2}
= - \frac12,
$$
so that by \eqref{Vastmu:m} we conclude that $(W_\alpha\ast
\mu_\alpha)(x) =  -\frac12 |x|^2 + C_\alpha$ for every
$x\in\Omega(\sqrt{1-\alpha},\sqrt{1+\alpha})$ with
$$
C_\alpha= -
\log\Big(\frac{\sqrt{1-\alpha}+\sqrt{1+\alpha}}2\Big)+\frac12 +
\alpha\, \frac{\sqrt{1-\alpha}}{\sqrt{1-\alpha}+\sqrt{1+\alpha}}=
2I_\alpha(\mu_\alpha)- \frac12\int_{\mathbb{R}^2} |x|^2
\,d\mu_\alpha(x).
$$
In particular, by \eqref{energy:ellipse} and \eqref{confinement:e}
we obtain the minimum value $I_\alpha(\mu_\alpha)$ of the energy
$I_\alpha$, that is,
$$
I_\alpha(\mu_\alpha) = \frac38 -\frac12
\log\Big(\frac{\sqrt{1-\alpha}+\sqrt{1+\alpha}}2\Big) +
\frac{\alpha}2\,
\frac{\sqrt{1-\alpha}}{\sqrt{1-\alpha}+\sqrt{1+\alpha}}.
$$

\begin{remark}[Computation of $W_\alpha\ast \mu_{a,b}$]\label{rmk:pot-computation}
For completeness, we can compute, for any $b\geq a>0$ and any $x\in
\mathbb{R}^2$, the value of $(W_{\alpha}\ast \mu_{a,b})(x)$.
Equation \eqref{Vastmu:m} gives
$$
(W_\alpha\ast \mu_{a,b})(x) = \frac{-a-b+\alpha \,b}{a(a+b)^2}\,
x_1^2 -  \frac{a+b+\alpha \,a}{b(a+b)^2}\, x_2^2 -
\log\Big(\frac{a+b}2\Big)+\frac12 + \alpha\, \frac{a}{a+b}
$$
for any $x\in \Omega(a,b)$. Outside the ellipse, it is again
convenient to pass to complex variables. Integrating
\eqref{gradoutside} with respect to $\bar{z}$, and recalling
\eqref{nablaH0} and \eqref{nablaH}, we can show that
$$
(W_\alpha\ast \mu_{a,b})(z) = H(z)+\alpha \Re\big(h(z)\bar{z}-ab
h(\bar{z})^2-\lambda h(\bar{z})\bar{z}\big)+\alpha \frac a{a+b}
$$
for any $z\in \Omega(a,b)^c$.
\end{remark}

\subsection{More general anisotropy}
Now we briefly discuss the case of a more general anisotropy of the
type
\begin{equation}\label{anisotropy-gen}
V_{\alpha,\beta,\gamma}(x_1,x_2) := \frac{\alpha x_1^2+\beta
x_2^2+\gamma x_1x_2}{x_1^2+x_2^2},
\end{equation}
where $\alpha,\beta,\gamma\in\R$. Let $W_\alpha^\gamma$ be the
kernel defined as
$$
W_{\alpha,\beta,\gamma}(x_1,x_2):=-\log|x| + \frac{\alpha
x_1^2+\beta x_2^2+\gamma x_1x_2}{x_1^2+x_2^2},
$$
and let $I_{\alpha,\beta,\gamma}$ be the corresponding energy,
defined on probability measures $\mu \in \P(\R^2)$ as
$$
I_{\alpha,\beta,\gamma}(\mu):= \frac12 \iint_{\mathbb{R}^2\times
\mathbb{R}^2} W_{\alpha,\beta,\gamma}(x-y) \,d\mu(x) \,d\mu(y) +
\frac12\int_{\mathbb{R}^2} |x|^2 \,d\mu(x).
$$

If $\gamma=0$, the anisotropy \eqref{anisotropy-gen} can be written
as
$$
V_{\alpha,\beta,0}(x_1,x_2)= (\alpha-\beta)\,
\frac{x_1^2}{x_1^2+x_2^2} +\beta,
$$
so that the study of minimisers of $I_{\alpha,\beta,0}$ is covered
by the previous analysis.

Assume $\gamma\neq0$. Consider the rotation in the plane defined by
$$
y=\frac{1}{\sqrt{a^2+\gamma^2}} \left(\begin{matrix} -a & \gamma
\\
-\gamma & -a
\end{matrix}\right) x
$$
where $a:=\beta-\alpha - \sqrt{(\beta-\alpha)^2+\gamma^2}$. Setting
$b:=\sqrt{(\beta-\alpha)^2+\gamma^2}$, a simple computation shows
that
$$
V_{\alpha,\beta,\gamma}(x_1,x_2) =  b\, \frac{y_1^2}{y_1^2+y_2^2} +
\beta -\frac{b\,\gamma^2}{a^2+\gamma^2},
$$
so that, up to this change of variables, the study of the minimality
of $I_{\alpha,\beta,\gamma}$ reduces again to the original case. In
particular, for $b<1$ the minimiser is an ellipse with major axis
along the line $y_1=0$, that is, $-ax_1+\gamma x_2=0$, while for
$b\geq1$ the minimiser is the semi-circle law on that line.

The two orthogonal lines $y_1=0$ and $y_2=0$ are the zero set of the
anisotropic force $F_{\alpha,\beta,\gamma}$, given by
\begin{equation}\label{force_a}
F_{\alpha,\beta,\gamma}(x) = - \nabla V_{\alpha,\beta,\gamma}(x) =
\frac{\gamma x_1^2- \gamma x_2^2+2(\beta-\alpha) x_1
x_2}{(x_1^2+x_2^2)^2} \, x^\perp,
\end{equation}
where $x^\perp=(x_2,-x_1)$. The force $F_{\alpha,\beta,\gamma}$ is
perpendicular to the radial direction, and it is indeed zero only
when
$$
x_2 = \frac1\gamma \big(\beta-\alpha \pm
\sqrt{(\beta-\alpha)^2+\gamma^2}\big) x_1,
$$
which correspond to $y_1=0$ and $y_2=0$. Looking at the sign of the
force in \eqref{force_a} it is clear that $F_{\alpha,\beta,\gamma}$
points towards the line $y_1=0$.

\end{section}

\bigskip\bigskip

\noindent \textbf{Acknowledgements.} JAC was partially supported by
the Royal Society via a Wolfson Research Merit Award and by
the EPSRC under the Grant EP/P031587/1. MGM and LR are partly
supported by GNAMPA--INdAM. MGM acknowledges support by the European
Research Council under Grant No.\ 290888. LR acknowledges support by
the Universit\`a di Trieste through FRA~2016. LS acknowledges
support by the EPSRC under the Grant EP/N035631/1. JM and JV
acknowledge support by the Spanish projects MTM2013-44699 (MINECO)
and MTM2016-75390 (MINECO), 2014SGR75 (Generalitat de Catalunya) and
FP7-607647 (European Union).

\bigskip

%\bibliography{dislocations}
%\bibliographystyle{abbrv}

\end{document}